\documentclass[12pt]{amsart}

\usepackage[utf8]{inputenc}
\usepackage{amsmath,amssymb,bbm,enumitem}
\usepackage{hyperref}
\usepackage{amsthm}
\usepackage{listings}[2013/08/05]
\usepackage[capitalize]{cleveref}
\usepackage[square,numbers,sort]{natbib}
\usepackage{xparse}

\hypersetup{pdfstartview=}

\def\clap#1{\hbox to 0pt{\hss#1\hss}}

\newcommand{\rhup}{\rightharpoonup}

\def\B{{\mathcal B}}
\def\D{{\mathcal D}}

\def\P{{\mathcal P}}

\def\p{{\mathfrak p}}

\def\k{{\mathbbm k}}

\def\ot{\otimes}
\newcommand\inv{^{-1}}
\def\iff{\Leftrightarrow}

\newcommand{\du}[1]{\mathbbm{k}^{#1}}
\newcommand{\duc}[1]{\mathbbm{k}^{#1\operatorname{co}}}
\newcommand\rcofix[2]{{#1}{^{\operatorname{co} #2}}}

\newcommand\comm\curlyvee
\newcommand\cocomm\curlywedge

\newcommand{\com}[1]{_{(#1)}}

\newcommand{\ev}[1]{\operatorname{ev}_{#1}}
\newcommand{\BCh}[1]{\Hom(#1,\widehat{#1})}

\DeclareMathOperator{\image}{Im}  
\DeclareMathOperator{\Img}{\image}
\DeclareMathOperator{\Hom}{Hom}   
\DeclareMathOperator{\id}{id}     
\DeclareMathOperator{\Aut}{Aut}   
\DeclareMathOperator{\End}{End}   
\DeclareMathOperator{\Rep}{Rep}		
\DeclareMathOperator{\co}{co}     
\DeclareMathOperator{\op}{op}			

\theoremstyle{plain}
\newtheorem{thm}{Theorem}[section]

\newtheorem{cor}[thm]{Corollary}
\newtheorem{prop}[thm]{Proposition}
\newtheorem{lem}[thm]{Lemma}

\theoremstyle{definition}
\newtheorem{df}[thm]{Definition}
\newtheorem{example}[thm]{Example}

\theoremstyle{remark}
\newtheorem{rem}[thm]{Remark}

\crefname{lem}{Lemma}{Lemmas}
\crefname{thm}{Theorem}{Theorems}
\crefname{cor}{Corollary}{Corollaries}
\crefname{prop}{Proposition}{Propositions}
\crefname{example}{example}{examples}
\crefname{df}{Definition}{Definitions}
\crefname{equation}{equation}{equations}

\numberwithin{equation}{section}

\DeclareDocumentCommand{\morph}{ O{u} O{r} O{p} O{v} }{\begin{pmatrix} #1 & #2 \\ #3 & #4 \end{pmatrix}}
\DeclareDocumentCommand{\cmorph}{ O{u} O{r} O{p} O{v} }{(#1,#2,#3,#4)}

\newcommand{\DDG}{\D(G)^{*\co}}
\DeclareDocumentCommand{\twocycle}{ O{u} O{r} O{p} O{v} O{G}}{\sum_{g,h,s,t\in #5} #2^*(h)e_s\# gt \ot e_h #1(e_g)\# #4(s)#3(e_t)}

\title[Quasitriangular structures for group doubles]{Quasitriangular structures of the double of a finite group}
\author{Marc Keilberg}
\thanks{A portion of this research was conducted during the author's stay at the Institut de Math\'{e}matiques de Bourgogne, Universit\'{e} de Bourgogne in Dijon, France, and was partially supported through a FABER Grant by the Conseil r\'{e}gional de Bourgogne.}
\email{keilberg@usc.edu}

\begin{document}
\begin{abstract}
We give a classification of all quasitriangular structures and ribbon elements of $\D(G)$ explicitly in terms of group homomorphisms and central subgroups.  This can equivalently be interpreted as an explicit description of all braidings with which the tensor category $\Rep(\D(G))$ can be endowed.  We also characterize their equivalence classes under the action of $\Aut(\D(G))$ and determine when they are factorizable.
\end{abstract}
\maketitle
\section*{Introduction}\label{sec:intro}
Quasitriangular (quasi-)bialgebras were first introduced by \citet{Dri:ACoCHA,Dri:QG,Drin:SQTqHA} as a way of producing solutions to the quantum Yang-Baxter equation.  These have applications in statistical mechanics, where they yield exactly solvable lattice models \citep{Jim:YBEIS}, as well as in quantum computing, where they can be characterized as universal quantum gates \citep{ZhaKau:YBUQG,KauLom:BOUQG}.  A number of knot and link invariants can also be constructed from such objects \citep{Res:QTHALinks,Kau:IL3MHA}. When we have the additional structure of a factorizable ribbon Hopf algebra, we also obtain projective representations of mapping
class groups of compact oriented surfaces of arbitrary genus with a finite (possible empty) collection of marked boundary circles \citep{Ly:Ribbon}.

From a categorical point of view, if $H$ is a quasitriangular (quasi-)Hopf algebra over the field of complex numbers then $\Rep(H)$, the category of finite-dimensional representations of $H$, is a braided tensor category.  In particular, the braidings of $\Rep(H)$ are precisely given by the quasitriangular structures of $H$ \citep[Theorem 10.4.2]{Mon:HAAR}.  The braided fusion categories $\Rep(\D(G))$, the finite dimensional representations of the Drinfel'd double of a group, and more generally $\Rep(\D^\omega(G))$ \citep{DPR}, have been of substantial recent interest in their own right.  See \citep{NegNg16,NaNi,GMN,Ma1,Na07,Ost:MCDG} and the references therein.

In general a Hopf algebra can have many quasitriangular structures, and there is a characterization of them in terms of certain Hopf algebra morphisms due to \citet{Rad:MQHA}.  A number of equivalent characterizations of quasitriangular structures, also known as universal $R$-matrices, have been provided for various types of Hopf algebras \citep{Jiao:QTSwSP,WanZhan:CQTSTSCP,ZhaShuJia:QTSBHA}. In this paper we will provide a complete and explicit description of the quasitriangular structures of $\D(G)$ over an arbitrary field in terms of central subgroups and group homomorphisms.

Investigating the impact of changing the braiding of $\Rep(H)$ for a semisimple ribbon Hopf algebra $H$, or more generally for any braided fusion category, is expected to provide additional insights into the category. For example, one of the most important invariants for a spherical category are the Frobenius-Schur indicators \citep{NS10,NegNg16,BCNPR17}. On the one hand the higher indicators can be computed using only that $H$ is semisimple \citep{KSZ2}, and on the other they can expressed in a number of categorical ways, especially when the category is modular \citep{NS07a}. Thus it is natural to ask when a given quasitriangular structure yields a modular category; if inequivalent modular categories can be obtained; and what new data can be obtained about the indicators by comparing the categorical calculations when the braiding is changed.  Furthermore, the modular data is connected to many other invariants, such as the fusion rules via Verlinde's formula~\citep{BakKir:book}, and we can similarly question what new insights we obtain about these invariants. Answering some of these questions for $\Rep(\D(G))$ is itself a detailed enterprise, however. We will subsequently focus our attention on the group theoretical and classical Hopf algebra questions for this paper. The author intends to address several of these questions for $\Rep(\D(G))$ in a future paper.

We note that some of the results in this paper could be stated in greater generality than given. In particular, \cref{lem:ABM3.1} is just a specific instance of the general ansatz of splitting a given algebraic object into distinct pieces (not necessarily with the same structure) and using this decomposition to analyze the original object.  The lemma in particular can be used to give various generalizations of \cref{thm:morphisms} and \citep[Theorem 3.2]{ABM} to morphisms between various combinations of bismash (co)product Hopf algebras.  This in turn becomes a description of weak $R$-matrices and quasitriangular structures on bismash products in much the same fashion.  However, doing so in such generality results in lengthy lists of equations for which the author is unable to find meaningful structure or simplifications.  As such we have opted to restrict focus to the doubles of groups, where the equations take on a reasonably straightforward description in group theoretical terms.

The paper is structured as follows. In \cref{sec:prelim} we introduce the relevant notation and background. In \cref{sec:weak-Rs} we characterize all Hopf algebra morphisms $\DDG\to\D(G)$. As in \citep{MR1624475}, we call these the weak $R$-matrices.  In \cref{sec:group-com,sec:dual-com} we compute the equivalent conditions for a weak $R$-matrix to have the appropriate commutation relationship with the comultiplication.  In \cref{sec:qts} we combine these results to describe all quasitriangular structures of $\D(G)$.  We then show that a ribbon element exists for each quasitriangular structure and explicitly describe them in \cref{sec:ribbon}. \Cref{sec:equivalence} investigates the equivalence of quasitriangular structures under $\Aut(\D(G))$. Finally, \cref{sec:factor} determines when an arbitrary quasitriangular structure is factorizable.

\section{Preliminaries}\label{sec:prelim}
Our reference for the general theory of Hopf algebras will be \citep{Mon:HAAR}.

Let $H$ be a Hopf algebra over a field $\k$.  Suppose that $R\in H\ot H$ satisfies the following relations:
\begin{align}
	(\Delta\ot\id)(R)&=R^{13}R^{23}\label{eq:qt1};\\
	(\id\ot\Delta)(R)&=R^{13}R^{12}\label{eq:qt2};\\
	(\varepsilon\ot\id)(R)&=1;\label{eq:qt3}\\
	(\id\ot\varepsilon)(R)&=1,\label{eq:qt4}
\end{align}
where, writing $R=\sum R^{(1)}\ot R^{(2)}$ we have $R^{13}=R^{(1)}\ot1\ot R^{(2)}$, and similarly for $R^{23}$ and $R^{12}$.  Any such element is invertible, with $R^{-1}=S\ot\id(R)$. Such an element is called a weak $R$-matrix on $H$ \citep{MR1624475}.

Furthermore, if a weak $R$-matrix $R$ satisfies \[h\com 2\ot h\com1 = R(h\com1\ot h\com2)R^{-1}\] for all $h\in H$ then $R$ is said to be a quasitriangular structure, or (universal) $R$-matrix, of $H$.  If such an element exists, we say that $H$ is quasitriangular, and denote the pair by $(H,R)$, or simply $H$ when the structure is understood from the context.
\begin{df}\label{df:equiv-QTS}
Two quasitriangular Hopf algebras $(H,R)$, $(K,R')$ are said to be isomorphic as quasitriangular Hopf algebras if there is a Hopf algebra isomorphism $X\colon H\to K$ such that $X\ot X(R)=R'$.  Given two quasitriangular structures $R,R'$ on $H$, we say that $R$ and $R'$ are equivalent, denoted $R\sim R'$, if $(H,R)$ and $(H,R')$ are isomorphic as quasitriangular Hopf algebras.
\end{df}

As noted by \citet{Rad:QSHA} there is a $\mathbbm{k}$-linear injection $F\colon H\ot H\to \Hom_{\mathbbm{k}}(H^*,H)$ given by $F(a\ot b)(p)=p(a)b$.  In the subsequent, $F$ will always refer to this injection.  When $H$ is finite dimensional then $R\in H\ot H$ satisfies equations \eqref{eq:qt1} to \eqref{eq:qt4} if and only if $F(R)$ is a morphism of Hopf algebras $\rcofix{H^*}{}\to H$.  Indeed, there are always finitely many such $R$ under mild assumptions on $H$ and $\k$.  On the other hand, given a morphism of Hopf algebras $\psi\colon \rcofix{H^*}{}\to H$ and any basis $\B$ of $H$ then
\begin{align}\label{eq:morph-to-weak}
	R=\sum_{h\in \B} h\ot\psi(h^*),
\end{align}
where $h^*$ is the element dual to $h$, satisfies $F(R)=\psi$.  Thus $R$ is a weak $R$-matrix, and by injectivity of $F$ it is independent of the choice of basis of $H$.

Let $K,L$ also be Hopf algebras over $\k$.  Given linear maps $f\colon H\to K$ and $g\colon L\to K$, by $f\comm g$ we mean that the images of $f$ and $g$ commute elementwise, and say that $f$ and $g$ commute.  Dually, for $f\colon H\to K$ and $g\colon H\to L$, by $f\cocomm g$ we mean that the morphisms cocommute: $f(a\com 1)\ot g(a\com 2) = f(a\com 2)\ot g(a\com 1)$ for all $a\in H$.  Throughout $\tau$ denotes the map $H\ot K\to K\ot H$ given by $\tau(h\ot k)=k\ot h$.

Any linear map $f\colon H\to K$ will be called unitary if $f(1_H)=1_K$, and counitary if $\varepsilon_K\circ f = f\circ\varepsilon_H$.  We say that $f$ is biunitary if it is both unitary and counitary.  All algebra morphisms are unitary, and all coalgebra maps are counitary, so we will not specify unitary or counitary in these cases.  A counitary algebra morphism is also called a morphism of augmented algebras.  All morphisms and spaces of morphisms will be of Hopf algebras or groups as appropriate, unless otherwise specified.

For a finite group $G$ we let $\k G$ be its group algebra over $\k$ and $\du{G}$ be the dual Hopf algebra.  The group of 1-dimensional $\k$-linear characters of $G$ is denoted by $\widehat{G}$, and is identified with the group-likes of $\du{G}$.  We denote the left conjugation actions of $G$ on $\k G$ and $\du{G}$ both by $\rhup$.  For $g,h\in G$ we let $g^h=h\inv g h$ and $[g,h]=g\inv h\inv g h$.  Note that $G^{\operatorname{op}}\cong G$ via the inversion map.  We say that $G$ is purely non-abelian if it has no non-trivial abelian direct factors.  A special case of such groups are the stem groups, which are those $G$ satisfying $Z(G)\subseteq G'$.

We now describe $\D(G)$, the Drinfel'd double of $G$ over $\k$.  As a coalgebra this is $\duc{G}\ot\k G$.  Denoting elements of $\D(G)$ by $f\bowtie g$, $f\in\duc{G}$, $g\in G$, the algebra structure is given by the semidirect product formula
\[ (f\bowtie g)\cdot (f'\bowtie g') = f (g\rhup f')\bowtie gg'.\]
Similarly, $\DDG$ is $\du{G}\ot \k G^{\op}$ as an algebra.  Denoting elements of $\DDG$ by $f\# g$, $f\in\du{G},g\in G^{\op}$, the coalgebra structure is given by
\[\Delta(e_x\# g) = \sum_{s\in G} e_s\# g \ot e_{xs^{-1}}\# s^{-1}gs.\]

Note that the conjugate $s^{-1}gs$ is computed in $G^{\op}$.  In particular we see that $\duc{G}$ is a Hopf subalgebra of $\DDG$, whereas $G^{\op}$ is only an augmented subalgebra since $\Delta(\varepsilon\# g) = \sum_{s\in G}e_s\# g \ot \varepsilon\# s^{-1}gs$.  For more details on these Hopf algebras, we refer the reader to \citep{DPR,Mon:HAAR,BGM}.

\begin{example}
When $H=\D(G)$ is the Drinfel'd double of a finite group, the standard quasitriangular structure is
\begin{align}\label{eq:dg-std-qt}
	R_0 = \sum_{g\in G}\varepsilon\# g\ot e_g\# 1.
\end{align}
We also have the following quasitriangular structure, which is sometimes used instead of $R_0$ depending on the choice of notation:
\begin{align}
  R_1=\tau(R_0\inv)=\sum_{g\in G}e_g\# 1\ot \varepsilon\# g\inv.
\end{align}
\end{example}

\section{Classifying weak \texorpdfstring{$R$}{R}-matrices for \texorpdfstring{$\D(G)$}{D(G)}}\label{sec:weak-Rs}
We wish to give a useful description of $\Hom(\DDG,\D(G))$, which then gives us a complete description of the weak $R$-matrices by \eqref{eq:morph-to-weak}.  This will involve a number of computations, so we state the result now and proceed to prove it in stages.  The idea is similar to that used by \citet{ABM} to classify the morphisms between bismash products of Hopf algebras.  $\D(G)$ is of this form, but $\DDG$ is a smash coproduct, so we must develop an appropriate version for the present context.

\begin{thm}\label{thm:morphisms}
	The morphisms $\psi\in\Hom(\DDG,\D(G))$ are in bijective correspondence with the quadruples $\cmorph$ where
	\begin{enumerate}
		\item $u\colon\duc{G}\to\duc{G}$ is a unitary morphism of coalgebras;
		\item $r\colon \k G^{\op}\to \duc{G}$ is a biunitary linear map;
		\item $p\colon \duc{G}\to \k G$ is a morphism of Hopf algebras;
		\item $v\colon \k G^{\op}\to \k G$ is a morphism of augmented algebras;
	\end{enumerate}
	satisfying all of the following, for all $a,b\in\duc{G}$ and $g,h\in G^{\op}$:
	\begin{align}
				p&\cocomm u;\\
				p&\comm v;\\
				u(ab)&=u(a\com 1)(p(a\com2)\rhup u(b));\label{compat:u-alg}\\
				r(gh)&=\sum_{s\in G}r(s\inv gs)(p(e_s)v(g)\rhup r(h));\label{compat:r-alg}\\
				\Delta(v(g))&= \sum_{s\in G}p(e_s)v(g)\ot v(s^{-1}gs);\label{compat:v-coalg}\\
				\Delta(r(g))&= \sum_{b,s\in G}u(e_s) \left(p(e_{bs^{-1}})\rhup r(g)\right) \ot r(b^{-1}gb); \label{compat:r-coalg}\\
				u(a\com 1)\left( p(a\com 2)\rhup r(g)\right) &= \sum_{s\in G}r(s\inv g s)\left( \left(p(e_s)v(g)\right)\rhup u(a) \right). \label{compat:ab-commute}
	\end{align}
	For such a quadruple, the morphism is given by
	\begin{align}\label{eq:morph-def}
		e_x\# g \mapsto \sum_{\substack{a,b,c\in G\\abc=x}} u(e_c)(p(e_b)\rhup r(a\inv g a))\bowtie p(e_a)v(g).
	\end{align}
	
	On the other hand, given any linear map $\psi\colon \D(G)^{*\operatorname{co}}\to \D(G)$, we define the components $u,r,p,v$ by defining
	\begin{align}
		u(a) &= \id\ot\varepsilon(\psi(a\#1));\label{abm:u-def}\\
		r(g) &= \id\ot\varepsilon(\psi(\varepsilon\# g));\label{abm:r-def}\\
		p(a) &= \ev1\ot\id(\psi(a\# 1));\label{abm:p-def}\\
		v(g) &= \ev1\ot\id(\psi(\varepsilon\# g)).\label{abm:v-def}
	\end{align}
\end{thm}

We use the notation of the theorem throughout the rest of the paper without further mention.  In particular we implicitly identify a morphism with its quadruple, adding indices or superscripts to the components to identify the particular morphism as necessary.  We will denote trivial morphisms by 0 and identity morphisms by 1.

It is easy to see, as in \citep[Theorem 2.1]{K14}, that the component $p$ is uniquely determined by a Hopf isomorphism $\du{A}\to \k B$, where $A,B$ are abelian subgroups of $G$. Subsequently we have isomorphisms $\widehat{A}\cong A\cong B$.  Whenever we mention $A,B$ in the subsequent we are referring to these subgroups.

We now proceed to prove the theorem.  We will show how to obtain the desired quadruple of maps and compatibility conditions from $\psi\in\Hom(\DDG,\D(G))$.  The reverse direction is then a simple check.  We need the following lemma, a proof of which can be found in \citep[Lemma 3.1]{ABM}.

\begin{lem}\label{lem:ABM3.1}
	Let $C,D,E$ be coalgebras and $H,K,L$ algebras.
	\begin{enumerate}
		\item There is a bijection between coalgebra morphisms $\psi\colon C\to D\ot E$ and pairs $(\gamma,\delta)$ where $\gamma\colon C\to D$ and $\delta\colon D\to E$ are cocommuting morphisms of coalgebras.  In particular, $\psi(c)=\gamma(c\com 1)\ot \delta(c\com 2)$, $u=(\id\ot\varepsilon)\psi$, and $p=(\varepsilon\ot\id)\psi$.
		\item There is a bijection between algebra morphisms $\phi\colon H\ot K\to L$ and pairs $(\alpha,\beta)$ where $\alpha\colon H\to L$ and $\beta\colon K\to L$ are commuting morphisms of algebras.  In particular, $\phi(h\ot k)=\alpha(h)\beta(k)$, $\alpha(h)=\phi(h\ot 1)$, and $\beta(k)=\phi(1\ot k)$.
	\end{enumerate}
\end{lem}

So suppose we are given $\psi\in\Hom(\DDG,\D(G))$.  We have that $\D(G)$ is a tensor product as a coalgebra, and subsequently $\DDG$ is a tensor product as an algebra.  Thus both parts of the lemma apply, and we may write $\psi(f\bowtie g) = \alpha(f)\beta(g) = \gamma((f\bowtie g)\com1)\ot \delta((f\bowtie g)\com2)$.  Furthermore, it is easily seen that $\alpha,\beta$ preserve the counit and that $\gamma,\delta$ preserve the unit.

In addition $\ev1\ot\id\colon\D(G)\to\k G$ is a morphism of Hopf algebras, whence we conclude that $\delta$ is in fact a morphism of Hopf algebras.  Applying the lemma again, we may write $\delta(f\bowtie g) = p(f)v(g)$ for $p\colon\duc{G}\to\k G$ a morphism of Hopf algebras and $v\colon\k G^{\op}\to\k G$ a morphism of augmented algebras satisfying $p\comm v$.  Since $\duc{G}$ is a Hopf subalgebra of $\DDG$ we also have that $\alpha$ is a morphism of Hopf algebras.  Therefore $\alpha(f) = u(f\com 1)\# p'(f\com 2)$ for $u\colon\duc{G}\to\duc{G}$ a morphism of unitary coalgebras and $p'\colon \duc{G}\to\k G$ a morphism of Hopf algebras satisfying $p'\cocomm u$.  Indeed \[\ev1\ot\id\left( \psi(f\bowtie 1)\right)=p(f)=p'(f)\] for all $f\in\duc{G}$.  We define $r(g)=\gamma(\varepsilon\# g)$ for all $g\in G^{\op}$.  This yields the quadruple $\cmorph$ in the theorem.  We now need to prove that the indicated compatibility relations hold, and that $\psi$ has the indicated form.

We first show that we can write $\beta$ in terms of $r,p,v$.  Since $\gamma\cocomm\delta$ we have
\begin{align*}
	\beta(g) &= \gamma\ot\delta(\Delta(\varepsilon\ot g))\\
	&= \gamma\ot\delta\left(\sum_{s\in G} \varepsilon\# s^{-1}gs\ot e_s\# g\right)\\
	&= \sum_{s\in G}r(s^{-1}gs)\# p(e_s)v(g).
\end{align*}
Subsequently we have \[\gamma(f\# g) = \id\ot\varepsilon(\alpha(f)\beta(g)) = u(f\com 1)(p(f\com 2)\rhup r(g)).\]  By then computing $\psi(f\# g)=\alpha(f)\beta(g)$ we find that \eqref{eq:morph-def} holds.

To get \eqref{compat:u-alg} we first observe that
\begin{align*}
	\alpha(f\cdot h) &= u(f\com 1 h\com 1)\bowtie p(f\com 2 h\com 2)\\
	&=\alpha(f)\alpha(h)\\
	&=u(f\com 1)\left(p(f\com 3)\rhup u(h\com 1)\right)\bowtie p(f\com 2)p(h\com 1).
\end{align*}
The desired relation then follows by applying $\id\ot\varepsilon$.

Similarly we have
\begin{align*}
	\beta(gh) &= \sum_{s\in G}r(s\inv ghs)\bowtie p(e_s)v(g)\\
	&= \beta(g)\beta(h)\\
	&= \sum_{s,t,x\in G}r(s\inv gs) \left( \left(p(e_x)v(g)\com 2\right)\rhup r( t\inv h t)\right)\\
    &\qquad{}\bowtie p(e_{sx\inv})v(g)\com1p(e_t)v(h).
\end{align*}
Applying $\id\ot\varepsilon$ we find that \eqref{compat:r-alg} holds.

We can also easily compute that
\begin{align*}
	\Delta\delta(\varepsilon\# g) &= \Delta(p(\varepsilon))\Delta(v(g))\\
	&=\Delta(v(g))\\
	&=\sum_{s\in G} p(e_s)v(g)\ot v(s\inv g s),
\end{align*}
which is \eqref{compat:v-coalg}.  By computing $\Delta\beta(g)$ in two different ways we similarly find that \eqref{compat:r-coalg} holds.

In order for $\alpha\comm \beta$ to hold we see that for all $f\in\duc{G},g\in G^{\op}$
\begin{align*}
	\sum_{s\in G}u(f\com 1)\left( p(f\com 3)\rhup r(s\inv g s)\right) \bowtie p(f\com 2 e_s)v(g)
\end{align*}
must be equal to
\begin{align*}
	\sum_{s,t\in G} r(s\inv g s)\left( p(e_t)v(g)\com1\rhup u(f\com1)\right)\bowtie p(e_{st\inv})v(g)\com2 p(f\com2).
\end{align*}
Applying $\id\ot\varepsilon$ to both expressions we find that \eqref{compat:ab-commute} holds.

This completes the proof.

By the bijective correspondence between the weak $R$-matrices and $\Hom(\DDG,\D(G))$ we have the following description of the weak $R$-matrices.

\begin{thm}
	Given $\cmorph{}\in\Hom(\DDG,\D(G))$ then
	\begin{align}\label{eq:R-def}
		R &= \sum_{a,b,c,s\in G} e_s\bowtie abc \ot u(e_c)\left( p(e_b)\rhup r(a\rhup s)\right)\bowtie p(e_a)v(s)
	\end{align}
	is a weak $R$-matrix with $F(R)=\cmorph$.
\end{thm}
\begin{rem}
	Expressing the comultiplication of $\DDG$ with the more general form of a semidirect coproduct would permit one to write $R$ in terms of arbitrary bases for $\duc{G}$ and $\k G$.  This does not provide a meaningful benefit in the subsequent, so we choose to express $R$ in the standard bases.
\end{rem}

\begin{example}\label{ex:standard-morph}
	The standard quasitriangular structure $R_0$ of $\D(G)$ in \eqref{eq:dg-std-qt} corresponds to the morphism $\cmorph[1][0][0][0]$.
\end{example}
\begin{example}
	For any morphism of Hopf algebras $u\colon\duc{G}\to\duc{G}$, $\cmorph[u][0][0][0]$ corresponds to the weak $R$-matrix \[R_u=\sum_{g\in G}\varepsilon\bowtie g\ot u(e_g)\bowtie 1.\]  When $u=\id$ we get the standard $R$-matrix.  Note that $\tau(R_u\inv)=(0,0,0,Su^*)$.
\end{example}
\begin{example}
	For any group homomorphism $r\colon G^{\op}\to\widehat{G}$, $\cmorph[0][r][0][0]$ gives the weak $R$-matrix \[R_r=\sum_{g\in G} e_g\bowtie 1\ot r(g)\bowtie 1.\]
\end{example}
\begin{example}
  For any Hopf algebra morphism $p\colon\duc{G}\to\k G$, $\cmorph[0][0][p][0]$ gives the weak $R$-matrix $R_p =\sum_{t\in G}\varepsilon\bowtie t\ot \varepsilon\bowtie p(e_t)$.
\end{example}
\begin{example}\label{ex:v-2cycle}
  For any $v\in\End(G)$, $\cmorph[0][0][0][Sv]$ gives the weak $R$-matrix \[R_v=\sum_{s\in G}e_s\bowtie 1\ot \varepsilon\bowtie v(s\inv).\]  Note that $\tau(R_v\inv)=\cmorph[v^*][0][0][0]$.  In particular for $v=\id$ we have $R_v=R_1=\tau(R_0\inv)$.
\end{example}

\section{Weak \texorpdfstring{$R$}{R}-matrices commuting with \texorpdfstring{$G$}{G}}\label{sec:group-com}
We wish to determine those weak $R$-matrices which commute with the coproduct, which we will call the central weak $R$-matrices, as well as the quasitriangular structures.  To determine the central weak $R$-matrices and quasitriangular structures of $\D(G)$ explicitly we need to check the equalities
\begin{equation}\label{eq:lazy}
	R\Delta(f\bowtie x) = \Delta(f\bowtie x)R
\end{equation}
and
\begin{equation}\label{eq:qt}
	R \Delta(f\bowtie x) = \tau\circ\Delta(f\bowtie x) R
\end{equation}
respectively.

It suffices to check these identities for elements of the form $\varepsilon\bowtie x$ and $f\bowtie 1$.  In this section we consider the former, and in the next section we consider the latter.

To this end we compute
\begin{align}\label{eq:Rx-prod}
	R\Delta(\varepsilon\bowtie x)&=\sum_{a,b,c,s\in G} e_s\bowtie abcx\\
    &\qquad{}\ot \left(p(e_b)\rhup r(a\rhup s)\right) u(e_c)\bowtie p(e_a)v(s)x\nonumber
\end{align}
and
\begin{align}\label{eq:xR-prod}
	(\Delta(\varepsilon\bowtie x))R&=\sum_{a,b,c,s\in G} x\rhup e_s\bowtie xabc\\
	&\qquad{}\qquad{}\ot x\rhup\left( (p(e_b)\rhup r(a\rhup s))u(e_c)\right)\nonumber\\
    &\qquad{}\qquad{} \bowtie x p(e_a)v(s).\nonumber
\end{align}

Applying $\ev1\ot\id\ot\ev1\ot\id$ to \eqref{eq:Rx-prod} and \eqref{eq:xR-prod} and equating we find
\[ \sum_{a\in G}xa\ot x p(e_a) = \sum_{a\in G} ax\ot p(e_a)x,\]
from which we conclude that $p(g\rhup f)=g\rhup p(f)$.  As a consequence $A\leq Z(G) \iff B\leq Z(G)$.  Note that $B\leq Z(G)$ implies that $F(R)(f\# g) = u(f\com 1)r(g)\bowtie p(f\com 2)v(g)$, thus simplifying \eqref{eq:morph-def}.

Applying $\ev1\ot\id\ot\id\ot\varepsilon$ instead we find
\[ \sum_{c\in G}cx\ot u(e_c) = \sum_{c\in G} xc\ot x\rhup u(e_c),\]
which is equivalent to $g\rhup u(f)=u(g\rhup f)$ for all $g\in G$, $f\in\duc{G}$.

Similarly, applying $\id\ot\varepsilon\ot\id\ot\varepsilon$ yields $r(h)=x\rhup r(x\rhup h)$, or equivalently that $x\rhup rS(h) = rS(x\rhup h)$. Lastly $\id\ot\varepsilon\ot\ev1\ot\id$ yields $g\rhup vS(h) = vS(g\rhup h)$.  Note that $vS\colon \k G\to\k G$ is a morphism of augmented algebras.

This proves necessity in the following, and the sufficiency is a simple check.
\begin{lem}\label{lem:G-commute}
	A weak $R$-matrix $R=\cmorph{}\in\D(G)\ot\D(G)$ satisfies $R\Delta(\varepsilon\# x) = \Delta(\varepsilon\# x)R$ for all $x\in G$ if and only if $u,rS,p,$ and $vS$ all commute with the conjugation actions of $G$.
\end{lem}

\section{Weak \texorpdfstring{$R$}{R}-matrices commuting with \texorpdfstring{$\Delta\du{G}$}{the comultiplication on the group dual}}\label{sec:dual-com}

We now check the equality of \eqref{eq:lazy} for elements of the form $f\bowtie 1$. To this end we compute
\begin{align}\label{eq:Rf-prod-lz}
	R \Delta(f\bowtie 1) &= \sum_{a,b,c,s,t\in G} e_s(abc\rhup f\com2)\bowtie abc\\
	&\qquad {} \ot \bigl(\bigl(p(e_t)v(s)\com1\bigr)\rhup f\com1\bigr)\nonumber\\
    &\qquad{}\qquad{}u(e_c)(p(e_b)\rhup r(a\rhup s))\nonumber \\
	&\qquad{}\qquad{}\bowtie p(e_{at\inv})v(s)\com2\nonumber
\end{align}
and
\begin{equation}\label{eq:fR-prod-lz}
\begin{split}
	(f\com2\bowtie 1\ot f\com1\bowtie 1)R &= \sum_{a,b,c,s\in G} f\com2 e_s\bowtie abc\\
	&\qquad {} \ot f\com1 \bigl( p(e_b)\rhup r(a\rhup s)\bigr)\\
    &\qquad{}\qquad{} u(e_c)\bowtie p(e_a)v(s).
\end{split}
\end{equation}

Applying $\ev1\ot\id\ot\id\ot\varepsilon$ to both expressions and equating we get
\[ \sum_{c\in G} c\ot f u(e_c) = \sum_{a,c\in G} ac\ot u (e_c)(p(e_a)\rhup f).\]
For the special case $f\in\Img(u)$ an application of \eqref{compat:u-alg} shows that $u$ is a morphism of Hopf algebras, from which it follows that we may identify $u^*\in\End(G)$.  Let $c=u^*(h)$ for some $h\in G$.  Then applying $\id\ot\ev{h}$ to the above equality we find
$(p(e_1)\rhup f)(h) = f(h)$.  This equation holds for all $h\in G$ if and only if $B\leq Z(G)$.

Applying $\id\ot\varepsilon\ot\ev1\ot\id$ to \eqref{eq:Rf-prod-lz} and \eqref{eq:fR-prod-lz} and equating we find
\[ \sum_{s\in G}f e_s\ot v(s) = \sum_{s,a\in G} e_s(a\rhup f)\ot p(e_a)v(s).\]
Thus for all $s\in G$ we have
\[ e_s f\ot 1 = \sum_{a\in G}e_s (a\rhup f)\ot p(e_a).\]
This forces $A\leq Z(G)$.  Then from \eqref{compat:v-coalg} we conclude that $v$ is a group homomorphism.  Similarly, \eqref{compat:r-alg} becomes
\begin{align}\label{compat:r-alg2}
	r(gh) = r(g)(v(g)\rhup r(h)).
\end{align}
Since $r$ is unitary we conclude that $r(g)$ is invertible for all $g\in G$.  Subsequently, \eqref{compat:ab-commute} simplifies to
\begin{align}\label{compat:ab-commute2}
	v(g)\rhup u(a)=u(a).
\end{align}

Now applying $\id\ot\varepsilon\ot\id\ot\varepsilon$ and equating we have
\begin{align}\label{eq:uisom-vcentral}
 \sum_{s\in G} f\com 2 e_s\ot f\com1 r(s) &= \sum_{c,s\in G} e_s(c\rhup f\com2)\\
 &\qquad{}\qquad{}\ot (v(s)\rhup f\com1)u(e_c)r(s).\nonumber
\end{align}
In particular for all $s,h\in G$
\[f\com 2 e_s\ot f\com1 r(s)e_h = \sum_{c\in G} e_s (u^*(h)\rhup f\com2)\ot r(s)e_h \left( v(s)\rhup f\com1\right).\]
Therefore for any fixed $s,h\in G$ we have
\[ r(s)(h)f\com2(s)f\com1(h)=r(s)(h)(u^*(h)\rhup f\com2)(s) (v(s)\rhup f\com1)(h)\]
which is equivalent to
\[ r(s)(h)f(hs) = r(s)(h) f(h^{v(s)}s^{u^*(h)}).\]
Since $r(s)$ is invertible, $r(s)(h)\neq 0$.  The arbitrary choice of $f$ then makes this equation equivalent to
\begin{equation}\label{eq:lazy-valpha}
 h^{v(s)}s^{u^*(h)} = hs
\end{equation}
for all $s,h\in G$.

The relation obtained by applying $\ev1\ot\id\ot\ev1\ot\id$ is trivially true in all cases.  This proves necessity in the following, with sufficiency being a simple check.

\begin{lem}\label{lem:dual-commute-lazy}
	A weak $R$-matrix $\cmorph{}\in\D(G)\ot\D(G)$ satisfies $R\Delta(f\bowtie 1) = \Delta(f\bowtie 1) R$ for all $f\in\duc{G}$ if and only if the following all hold:
	\begin{enumerate}
		\item $u$ is a morphism of Hopf algebras, or equivalently $u^*\in \End(G)$;
		\item $A,B\leq Z(G)$;
		\item $v$ is a morphism of Hopf algebras;
		\item $v(g)\rhup u(a)=u(a)$ for all $a\in\duc{G},g\in G^{\op}$;
		\item Equation \eqref{eq:lazy-valpha} is satisfied for all $s,h\in G$.
	\end{enumerate}
\end{lem}

\begin{example}
	Any $u^*,v$ with central image clearly satisfy \eqref{eq:lazy-valpha}.  We will see later that these are the only possibilities for a central weak $R$-matrix.
\end{example}

Now when we consider \eqref{eq:qt}, instead, we easily observe that all of the preceding arguments apply, with the exception that \eqref{eq:lazy-valpha} is replaced with
\begin{align}\label{eq:qt-valpha}
	s^{u^*(h)} h^{v(s)}=hs.
\end{align}

\begin{lem}\label{lem:dual-commute-qt}
	A weak $R$-matrix $R=\cmorph{}\in\D(G)\ot\D(G)$ satisfies \[R\Delta(f\bowtie 1) = \tau(\Delta(f\bowtie 1))R\] for all $f\in\duc{G}$ if and only if the following all hold:
	\begin{enumerate}
		\item $u$ is a morphism of Hopf algebras, or equivalently $u^*\in\End(G)$;
		\item $A,B\leq Z(G)$;
		\item $v$ is a morphism of Hopf algebras;
		\item $v(g)\rhup u(a)=u(a)$ for all $a\in\duc{G},g\in G^{\op}$;
		\item Equation \eqref{eq:qt-valpha} is satisfied for all $s,h\in G$.
	\end{enumerate}
\end{lem}

\begin{example}
	If $u^*$ has central image and for all $s\in G^{\op}$ $v(s)=z s\inv$ for some central element $z$ depending on $s$, then \eqref{eq:qt-valpha} is clearly satisfied.  Conversely, if $v$ has central image and for all $s\in G$ $u^*(s)= z s$ for some central element $z$ depending on $s$, then once again \eqref{eq:qt-valpha} is satisfied.  We will see later that these are the only possibilities for a quasitriangular structure when $G$ is indecomposable, and that $u^*, v$ are naturally built from such examples on indecomposable factors otherwise.
\end{example}

\section{The central weak \texorpdfstring{$R$}{R}-matrices and quasitriangular structures}\label{sec:qts}
Having computed the commutation relations we can now easily give a precise description of the central weak $R$-matrices and the quasitriangular structures.

For any finite group $G$ we may use Krull-Schmidt to write $G=G_0\times G_1\times\cdots\times G_n$ where $G_0$ is abelian and $G_i$ is an indecomposable non-abelian group for all $1\leq i\leq n$.  Let $\iota_i,\pi_i$ be the corresponding injection and projection respectively for $G_i$, $0\leq i\leq n$.  For any endomorphism $w\colon G\to G$ define $w_{ij}=\pi_i\circ w\circ\iota_j$, and set $w_i=w_{ii}$.  The $w_{ij}$ are also endomorphisms of $G$ and uniquely determine $w$ \citep{BCM}.  We make the analogous description when $w\colon G^{\op}\to G$ as well.

\begin{prop}\label{prop:central-cycles}
	A weak $R$-matrix $\cmorph$ is a central weak $R$-matrix if and only if the following all hold:
	\begin{enumerate}
		\item $vS,u^*\in\Hom(G,Z(G))$;
		\item $r$ is a bicharacter, meaning $r\in\Hom(G^{\op},\widehat{G})=\Hom(G,\widehat{G})$;
		\item $A,B\leq Z(G)$.
	\end{enumerate}
	In this case the weak $R$-matrix may be written as
	\begin{align}\label{eq:central-weak}
        \sum_{a,c,s\in G} e_s\bowtie ac \ot r(s)u(e_c)\bowtie p(e_a)v(s).
    \end{align}
\end{prop}
\begin{proof}
	The only remaining case is to suppose $R$ is a central weak $R$-matrix and to show that $v,u^*\in\Hom(G,Z(G))$ follows from \eqref{eq:lazy-valpha}.  By \cref{lem:G-commute} $Sv,u^*$ are normal group endomorphisms $G\to G$.  Decompose $G$ and its endomorphisms as before.  It follows that without loss of generality we may consider \eqref{eq:lazy-valpha} under the assumption that $G$ is indecomposable and non-abelian.  We note that $u^*$ and $(Su^*)*\id$ are simultaneously normal group endomorphisms.  Therefore by normality of $u^*$ and assumptions on $G$ either $u^*$ or $(Su^*)*\id$ is a central automorphism.  In this case the other is necessarily in $\Hom(G,Z(G))$.  Similarly, either $vS$ is a central automorphism or $vS\in\Hom(G,Z(G))$.  It is easily checked that \eqref{eq:lazy-valpha} then holds if and only if $u^*,vS\in\Hom(G,Z(G))$, as desired.
\end{proof}
As a consequence we have the following.
\begin{cor}\label{cor:central-weak-group}
	Define $Z_\k(G)$ to be the maximal subgroup of $Z(G)$ all of whose subgroups are isomorphic to their character groups over $\k$.  Then the central weak $R$-matrices form an abelian group isomorphic to
	\[ \Hom(G,Z(G))^2\times \Hom(G,\widehat{G})\times \End(Z_\k(G)),\]
	where the multiplication of central weak $R$-matrices is given by componentwise convolution products.
\end{cor}

Using similar arguments to those in the proof of \cref{prop:central-cycles} we obtain the following explicit description of the quasitriangular structures of $\D(G)$.

\begin{thm}\label{thm:qts}
	A weak $R$-matrix $\cmorph$ is a quasitriangular structure of $\D(G)$ if and only if the following all hold:
	\begin{enumerate}
		\item $A,B\leq Z(G)$;
		\item $r\in\Hom(G,\widehat{G})$ is a bicharacter;
		\item $u^*,Sv$ are normal endomorphisms of $G$;\label{thm:qts-normal-condition}
		\item For each $1\leq i\leq n$ exactly one of the following holds: \label{thm:qts-autc-condition}
		\begin{enumerate}
			\item $Sv_i\in\Hom(G_i,Z(G_i))$, $u^*_i\in\Aut_c(G_i)$;
			\item $Sv_i\in\Aut_c(G_i)$, $u^*_i\in\Hom(G_i,Z(G_i))$.
		\end{enumerate}
	\end{enumerate}
	In this case we have
	\begin{align}\label{eq:qts-form}
        \cmorph{} = \sum_{a,b,s\in G} e_s\bowtie ab \ot r(s)u(e_b)\bowtie p(e_a)v(s).
    \end{align}
	
	In particular there is a $2^n$-to-1 correspondence between quasitriangular structures and the central weak $R$-matrices.  When $G$ is abelian the quasitriangular structures and central weak $R$-matrices are (trivially) the same.
\end{thm}
\begin{proof}
As noted, the calculations for the central weak $R$-matrices applies to this case as well, with much the same arguments showing that \eqref{eq:qt-valpha} yields the stated description of the components for $v$ and $u^*$.

For the last claim, suppose we are given a set $E\in\P(\{1,...,n\})$. Then for a central weak $R$-matrix $\cmorph$ we construct a quasitriangular structure $\cmorph[u'][r][p][v']$ by defining
\begin{align*}
	u'_{ij} &= \begin{cases}u_{ij}& i\neq j \mbox{ or } i\not\in E\\
	Su_i*\id& i=j \mbox{ and } i\in E;\end{cases}\\
	v'_{ij} &= \begin{cases} v_{ij}& i\neq j \mbox{ or } i\in E\\
	v_i*S & i=j \mbox{ and } i\not\in E.\end{cases}
\end{align*}
That reversing this process yields a central weak $R$-matrix follows from \citep{AY65}.
\end{proof}
Note that the correspondence is dependent upon the choice of decomposition of $G$.  Since the quasitriangular structures themselves are independent of this decomposition, we will simply leave a fixed but otherwise arbitrary choice of decomposition for $G$ implicit.  The quasitriangular structures associated to the sets $E=\emptyset$ and $E=\{1,...,n\}$, however, are canonically determined and do not depend on the choice of decomposition.  The quasitriangular structures obtained from the trivial weak $R$-matrix are the standard quasitriangular structures $R_0$ and $R_1=\tau(R_0\inv) = \sum_{g\in G}e_g\bowtie 1\ot \varepsilon\bowtie g\inv$ irrespectively.

\begin{rem}
	The last two conditions from the theorem can be restated as follows.  Let $\hat{v}\colon G/G_0\to G/G_0$ be given by $\hat{v}_{ij}=Sv_{ij}$ for $i,j>0$, and similarly for $\hat{u}$.  Then the last two conditions are equivalent to $\hat{v},\hat{u}^*$ being normal and $\hat{u}^**\hat{v}\in\Aut_c(G/G_0)$.  In other words, $\hat{v}$ and $\hat{u}^*$ give a convolution factorization of a central automorphism of $G/G_0$ into normal endomorphisms.  It is worth pointing out that neither $v_0$ nor $u_0$ need be an isomorphism, and indeed that $u^* * v$ need not be a central automorphism of $G$.
\end{rem}

\begin{example}\label{ex:transfers}
	For any quasitriangular structure $R=\cmorph$ associated to $E\in\P(\{1,...,n\})$ we can easily check that $\tau(R\inv)=(Sv^*,Sr^*,Sp^*,Su^*)$ is a quasitriangular structure associated to $E^c$.  Indeed, $R$ is obtained from the trivial central weak $R$-matrix if and only if $\tau(R\inv)$ is obtained from the trivial central weak $R$-matrix.
\end{example}

\section{Ribbon elements}\label{sec:ribbon}
We now recall the basic facts about ribbon Hopf algebras, which can be found in \citep{Rad:Book}. Given a quasitriangular Hopf algebra $(H,R)$, we define the Drinfel'd element to be $u_R = m(\tau(R\inv))$, where $m\colon H\otimes H\to H$ is the multiplication of $H$.  This element satisfies $S^2(h) = u h u\inv$ for all $h\in H$.  We say that $\nu\in H$ is a ribbon element of $(H,R)$ if $\nu^2 = u Su$, $\nu$ is central and invertible in $H$, $\varepsilon(\nu)=1$, $S\nu = \nu$, and
\[ \tau(R)R \ \Delta(\nu) = \nu\otimes \nu.\] When such a $\nu$ exists, we say that $(H,R,\nu)$, or just $(H,R)$ or $H$ when there is no ambiguity, is a ribbon Hopf algebra.

In general a ribbon element is not necessarily uniquely defined when it exists, but by taking the ratio of any two ribbon elements we see that they differ by multiplication by a central group-like element of $H$ that has order dividing 2.  In the case where $\D(G)=H$, the group-likes are precisely $\widehat{G}\times G$, which has center $\widehat{G}\times Z(G)$.

We will now show that $(\D(G),R)$ admits a ribbon element for any choice of quasitriangular structure.
\begin{thm}\label{thm:ribbon}
	Let $R=\cmorph$ be a quasitriangular structure of $\D(G)$. Then for the quasitriangular Hopf algebra $(\D(G),R)$ the Drinfel'd element is
		\[ u_R= \sum_{a,s\in G} r(s) e_{s\inv}\bowtie p(e_a)a\inv v(s)u^*(s).\]
	Furthermore, $u_R$ is also a ribbon element.  
\end{thm}
\begin{proof}
	Let $R=\cmorph{}$ be a quasitriangular structure.  By definition $u_R=m(\tau(R\inv))$.  We have
	\[ R\inv = \id\ot S(R) = \sum_{s,a,b\in G}e_{s\inv}\bowtie a\inv b\inv\ot r(s) u(e_b)\bowtie p(e_a)v(s),\]
	whence
	\begin{align*}
		u_R=m(\tau(R\inv)) &= \sum_{s,a,b\in G} r(s)u(e_b)(v(s)\rhup e_{s\inv})\bowtie p(e_a)v(s)a\inv b\inv\\
		&= \sum_{s,a,b\in G} r(s) u(e_b)e_{s\inv} \bowtie p(e_a)a\inv v(s)b\inv,
	\end{align*}
	where we have used that the components of $v$ are either central or central automorphisms, and $A,B\leq Z(G)$.  We then observe that $u(e_b)e_{s\inv}\neq 0$ if and only if $u^*(s\inv)=b$.  This gives the desired formula for $u_R$.
	
	Now since any component of $v$ which is not central implies that the same component of $u^*$ is central, and vice versa, we see that $u^*\comm v$.  Since all components of $u^*$ are also either central or central automorphisms, we conclude that $Su_R=u_R$.  Since $\D(G)$ is involutory $u_R$ must be central.

    The only relation $u_R$ must then satisfy which is non-trivial and not yet established is that $\tau(R)R \ \Delta(u_R)=u_R\otimes u_R$.  This is equivalent to $R \ \Delta(u_R) = \tau(R\inv) u_R\otimes u_R$. We will show this relation holds when $u^*\in\Aut_c(G)$ is a central isomorphism and $Sv\in\Hom(G,Z(G))$ is a central homomorphism.  A similar argument, which we omit, then shows that the relation also holds when $u^*\in\Hom(G,Z(G))$ and $Sv\in\Aut_c(G)$.  This establishes the result for an indecomposable non-abelian group, and the general case then follows by breaking $u,v$ into components.  For simplicity of performing calculations, we may suppose that $v$ has domain $G$ (rather than $G^{\op}$) by replacing it with $Sv$; meaning that all $v(s)$ terms in expressions will be replaced with $v(s\inv)$.

    For ease of reference, we recall the following identities.
    \begin{align}
      R &= \sum_{s,a,t\in G} e_s\bowtie au^*(t)\otimes r(s)e_t\bowtie p(e_a)v(s\inv);\label{eq:R-def-tmp}\\
      R\inv &= \sum_{s,a,t\in G} e_{s\inv}\bowtie a\inv u^*(t\inv)\otimes r(s)e_t\# p(e_a)v(s\inv);\label{eq:Rinv-def-tmp}\\
      \tau(R\inv) &= \sum_{s,a,t\in G} r(s)e_t\bowtie p(e_a)v(s\inv)\otimes e_{s\inv}\bowtie a\inv u^*(t\inv);\label{eq:tauRinv-def-tmp}\\
      u_R &= \sum_{y,b\in G} r(y\inv) e_y\bowtie p(e_b)b\inv v(y)u^*(y\inv).\label{eq:u-def-temp}
    \end{align}
    We then have
    \begin{align}\label{eq:Delta-uR}
    \begin{split}
      \Delta u_R &= \sum_{y,h,c,b\in G} r(y\inv) e_h\bowtie p(e_c)b\inv v(y)u^*(y\inv) \\
      &\quad\quad\otimes r(y\inv)e_{yh\inv}\bowtie p(e_{bc\inv})b\inv v(y)u^*(y\inv).
    \end{split}
    \end{align}
    We can then compute that
    \begin{align}
      R\Delta u_R &= \sum_{y,s,t,h,c,b,a\in G} r(y\inv) e_s e_{tht\inv} \bowtie p(e_c)b\inv a v(y)u^*(ty\inv)\nonumber\\
       &\qquad\qquad{}\otimes r(sy\inv) e_t e_{yh\inv} \bowtie p(e_a)p(e_{bc\inv}) b\inv v(s\inv y)u^*(y\inv)\nonumber\\
      &=\sum_{y,h,c,b\inv G} r(y\inv) e_{yhy\inv}\bowtie p(e_c)c\inv v(y)u^*(yh\inv y\inv)\nonumber\\
      &\qquad\qquad{}\otimes r(hy\inv) e_{yh\inv}\bowtie p(e_{bc\inv})b\inv v(h\inv y)u^*(y\inv)\nonumber\\
      \begin{split}\label{eq:R-delta-uR}
      &=\sum_{y,x,c,b\in G} r(y\inv) e_{x\inv} \bowtie p(e_c)c\inv v(y)u^*(x)\\
      &\qquad\qquad{}\otimes r(y\inv x\inv)e_{xy}\bowtie p(e_{bc\inv})b\inv v(xy)u^*(y\inv).
      \end{split}
    \end{align}

    Next we have
    \begin{equation}\label{eq:uR-tensor}
        \begin{split}
          u_R\otimes u_R &= \sum_{g,h,b,c\in G} r(g\inv) e_g\bowtie p(e_c)c\inv v(g)u^*(g\inv)\\
          &\qquad\qquad{}\otimes r(h\inv)e_h\bowtie p(e_b)b\inv v(h)u^*(h\inv),
        \end{split}
    \end{equation}
    and thus
    \begin{align}\label{eq:tau-uR}
        \tau(R\inv)& u_R\otimes u_R\nonumber\\
        &= \sum_{s,g,t,a,c,h,b\in G} r(sg\inv) e_t e_g\bowtie p(e_a)p(e_c)c\inv v(s\inv g)u^*(g\inv)\nonumber\\
        &\qquad\qquad{}\otimes r(h\inv)e_{s\inv} e_{t\inv ht}\bowtie p(e_b)b\inv a\inv v(h)u^*(t\inv h\inv)\nonumber\\
        &= \sum_{h,g,b,c\in G} r(h\inv g\inv) e_g \bowtie p(e_c)c\inv v(hg)u^*(g\inv)\nonumber\\
        &\qquad\qquad{}\otimes r(h\inv) e_{g\inv h g}\bowtie p(e_b)b\inv c\inv v(h)u^*(g\inv h\inv)\nonumber\\
        &= \sum_{h,g,b,c\in G} r(h\inv g) e_{g\inv} \bowtie p(e_c)c\inv v(hg\inv)u^*(g)\nonumber\\
        &\qquad\qquad{}\otimes r(h\inv) e_{ghg\inv} \bowtie p(e_b)b\inv c\inv v(hg)u^*(gh\inv)\nonumber\\
        &= \sum_{g,y,b,c\in G} r(y\inv) e_{g\inv} p(e_c)c\inv v(y)u^*(g)\nonumber\\
        &\qquad\qquad{}\otimes r(y\inv g\inv) e_{gy}\bowtie p(e_b)b\inv c\inv v(gy)u^*(y\inv)\nonumber\\
        \begin{split}
        &= \sum_{x,y,b,c\in G} r(y\inv) e_{x\inv} \bowtie p(e_c)c\inv v(y)u^*(x)\\
        &\qquad\qquad{}\otimes r(y\inv x\inv) e_{xy}\bowtie p(e_{bc\inv})c\inv v(xy)u^*(y\inv).
        \end{split}
    \end{align}
    This is precisely \cref{eq:R-delta-uR}, and so completes the proof that $u_R$ is a ribbon element.
\end{proof}
\begin{rem}
  An informal proof that $u_R$ is a ribbon element is as follows.  By \cref{thm:qts} the only meaningful impact the field has on a quasitriangular structure $R$ of $\D(G)$, and subsequently its Drinfel'd element $u_R$ and the desired identity $\tau(R)R \ \Delta(u_R)=u_R\otimes u_R$, is that it limits the choices for $p$ and $r$. Thus if the identity always holds in one field, it must hold in any field compatible with $p,r$.  Furthermore, it is well-known that any semisimple and cosemisimple quasitriangular Hopf algebra $H$ has its Drinfel'd element as a ribbon element.  Since $\D(G)$ always has these properties when $\k$ is algebraically closed with characteristic zero, $u_R$ must always be a ribbon element.
\end{rem}
\begin{example}
  For $(\D(G),R_0)$ we have $u_{R_0} = \sum_{g\in G} e_g\bowtie g$ is a ribbon element.  For $(\D(G),R_1)$ we have $u_{R_1} = \sum_{g\in G} e_g\bowtie g\inv$.  Both of these are well-known.
\end{example}

\section{Equivalence under \texorpdfstring{$\Aut(\D(G))$}{Aut(D(G))}}\label{sec:equivalence}
In this section we investigate the equivalence relation on quasitriangular structures given in \cref{df:equiv-QTS}.  We note that if $R\sim R'$ are two equivalent quasitriangular structures of $\D(G)$, then any isomorphism of the quasitriangular Hopf algebras $(\D(G),R)$ and $(\D(G),R')$ induces a braided equivalence of their representation categories.  In general, though, it is possible for non-isomorphic quasitriangular Hopf algebras to have representation categories which are equivalent as braided tensor categories. As has been mentioned, it is beyond the scope of this paper to settle equivalence at the categorical level, but in light of the conjectural connection between $\Aut(\D(G))$ and the structure of (braided) autoequivalences of $\Rep(\D(G))$ \citep{LenPri:Monoidal,LenPri:Brauer}, the classical picture may in fact resolve most, if not all, of the categorical one.

We first recall the fundamental properties of $\Aut(\D(G))$.
\begin{thm}\label{thm:aut-dg}\citep{K14,KS14,LenPri:Brauer}
  Every automorphism $X$ of $\D(G)$ can be described by a matrix
  \[ X=\begin{pmatrix}
    \alpha&\beta\\
    \gamma&\delta
  \end{pmatrix}\]
  where $\alpha^*,\delta$ are normal group homomorphisms, $\beta\in\BCh{G}$ is a bicharacter, and $\gamma\colon\du{G}\to\k G$ is a morphism of Hopf algebras associated to central subgroups $A,B$. The automorphism is explicitly given by
  \[ f\bowtie x\mapsto \beta(x) \alpha(f\com1)\bowtie \gamma(f\com 2)\delta(x)\]
  for all $f\in\duc{G}$ and $x\in G$.

  Any $X\in\Aut(\D(G))$ also satisfies $X^*\in\Aut(\D(G))$, where $X^*$ is the linear dual of $X$ given by
  \[ X^*=\begin{pmatrix}\delta^*&\beta^*\\\gamma^*&\alpha^*.\end{pmatrix}\]

  When $G$ is purely non-abelian then $\alpha^*,\delta\in\Aut(G)$ and $\delta\alpha^*\in\Aut_c(G)$. Moreover, in this case $\Aut(\D(G))$ consists precisely of those matrices satisfying these conditions on $\alpha,\beta,\gamma,$ and $\delta$.
\end{thm}

\begin{rem}\label{rem:R-as-mat}
  If $R=\cmorph$ is a central weak $R$-matrix or a quasitriangular structure for $\D(G)$, then we may similarly write $F(R)$ as a matrix
  \[\begin{pmatrix} u&r\\p&v\end{pmatrix}.\]
  These matrix descriptions for $R$ and $\Aut(\D(G))$ are compatible, in the sense that we may perform matrix multiplications in the usual fashion, but where multiplication is composition and addition is convolution product. Some examples of this are given in \cref{ex:structure-dupes} below.
\end{rem}

\begin{df}\label{df:isom-ribbon}
We say that ribbon Hopf algebras $(H,R,\nu), (K,R',\nu')$ are isomorphic as ribbon Hopf algebras if there exists an isomorphism of quasitriangular Hopf algebras $X\colon (H,R)\to (K,R')$ such that $X(\nu)=\nu'$.  In this case $X$ is called an isomorphism of ribbon Hopf algebras.
\end{df}

By \citep[Theorem 3.5]{K14} we have that $\D(G)\cong \D(H)$ for finite groups $G,H$ if and only if $G\cong H$. The following standard results then show that \cref{df:isom-ribbon} is essentially the same as \cref{df:equiv-QTS} when $H=K=\D(G)$.
\begin{lem}\label{lem:quasi-fixes-drin}
    Let $(H,R), (K,R)$ be two quasitriangular Hopf algebras with Drinfel'd elements $u_H,u_K$ respectively.  Suppose $X\colon (H,R)\to (K,R')$ is an isomorphism of quasitriangular Hopf algebras.  Then $X(u_H)=u_{K}$.
\end{lem}
\begin{proof}
  Let $X,R,R'$ be as in the statement with $X\ot X(R)=R'$.  Then
  \begin{align*}
    u_{K}&=m(\tau(R^{\prime-1}))\\&= m(\tau(X\ot X(R\inv)))\\&= m(X\ot X\tau(R\inv))\\
    &= X(m(\tau(R\inv)))\\
    &=X(u_H).
  \end{align*}
\end{proof}
\begin{cor}\label{cor:quasi-is-ribbon}
  If $R,R'$ are two quasitriangular structures of $\D(G)$ and $X\in\Aut(\D(G))$ is such that $X\ot X(R)=R'$, then $X$ is an isomorphism of ribbon Hopf algebras $(\D(G),R,u_R)\to (\D(G),R',u_{R'})$.
\end{cor}
\begin{proof}
  Apply the preceding lemma and \cref{thm:ribbon}.
\end{proof}

More generally, we can describe the action of $\Aut(\D(G))$ on the quasitriangular structures as follows.

\begin{thm}\label{prop:aut-qts}
	Let $R=\cmorph{},R'=(u',r',p',v')$ be quasitriangular structures of $\D(G)$, and let
    \[X=\begin{pmatrix} \alpha&\beta\\ \gamma&\delta\end{pmatrix}\in\Aut(\D(G)).\]
    Then the following are equivalent:
\begin{enumerate}
  \item $X\ot X(R)=R'$;
  \item $X\circ (F(R) \circ X^*) = F(R')$;
  \item The following four identities all hold:
	\begin{align}
		u'=& \alpha r\gamma^*+\alpha u \delta^* + \beta p \delta^* + \beta v \gamma^*;\label{eq:aut-qts-u}\\
		r'=& \alpha r \alpha^* + \alpha u \beta^*+\beta p \beta^*+\beta v \alpha^*;\label{eq:aut-qts-r}\\
		p'=& \gamma r\gamma^*+\gamma u\delta^*+\delta p\delta^*+\delta v\gamma^*;\label{eq:aut-qts-p}\\
		v'=& \gamma r\alpha^*+\gamma u\beta^*+\delta p \beta^*+\delta v\alpha^*,\label{eq:aut-qts-v}
	\end{align}
where addition denotes convolution product.
\end{enumerate}
\end{thm}
\begin{proof}
    Let $X, R, R'$ be as in the statement.  Then by \citep{KS14}
    \[ X^*=\begin{pmatrix} \delta^*&\beta^*\\ \gamma^*&\alpha^*\end{pmatrix}\in\Aut(\DDG)\]
    is given by $X^*(f\# s) = \sum_{t\in G} \beta^*(s)\delta^*(f\com1)\# \gamma^*(f\com2)\alpha^*(s)$ for all $f\#g\in\DDG$.

	Now $X\ot X(R)=R'$ is equivalent to
    \begin{align}\label{eq:equiv-as-func}
        F(R')(f\# h) &= \sum_{s,g\in G}\ev{h}\ot f(X(e_s\bowtie g))\, (X\circ F(R))(e_g\# s)
    \end{align}
    for all $f\# h\in\DDG$.  We then have
    \begin{align*}
      F(R')(\varepsilon\# h) &= \sum_{s,g\in G}\ev{h}(\beta(g)\alpha(e_s))\, (X\circ F(R))(e_g\# s)\\
      &= \sum_{g\in G}\beta(g,h)\, (X\circ F(R))(e_g\# \alpha^*(h))\\
      &= (X\circ F(R))(\beta^*(h)\# \alpha^*(h))\\
      &= (X\circ (F(R)\circ X^*))(\varepsilon\# h).
    \end{align*}
    Furthermore,
    \begin{align*}
      F(R')(f\# 1) &= \sum_{s,g\in G} f(\gamma(e_s) \delta(g)) (X\circ F(R))(e_g\# s)\\
      &= \sum_{s,g\in G} f\com2(\gamma(e_s))f\com1(\delta(g)) (X\circ F(R))(e_g\# s)\\
      &= X\circ F(R)(\delta^*(f\com1)\# \gamma^*(f\com2)).\\
      &= (X\circ (F(R)\circ X^*))(f\# 1).
    \end{align*}
    This shows that the first two items are equivalent.

    That the second item is equivalent to the third follows from \cref{thm:morphisms}; \cref{abm:p-def,abm:r-def,abm:u-def,abm:v-def} in particular.
\end{proof}
\begin{cor}
	Let $R,R'$ be quasitriangular structures of $\D(G)$ and $X\in\Aut(\D(G))$.
	Then $X\ot X(R)=R'$ if and only if $X\ot X(\tau(R\inv))=\tau(R^{\prime-1})$.
\end{cor}
\begin{proof}
	The result follows from \cref{ex:transfers,prop:aut-qts}.
\end{proof}

\begin{example}\label{ex:structure-dupes}
Consider the quasitriangular structures
\[ R=\cmorph[1][r][0][0] \mbox{ and } R'=\cmorph[1][0][p][0]\]
of $\D(G)$.  We have
  \begin{align*}
    X&=\begin{pmatrix} 1 & r^*\\0&1\end{pmatrix}\in\Aut(\D(G));\\
    Y&=\begin{pmatrix} 1& 0\\p&1\end{pmatrix}\in\Aut(\D(G)).
  \end{align*}
  Moreover, we can use the matrix notations to compute $X\circ(F(R_0)\circ X^*)$:
  \begin{align*}
    X(\begin{pmatrix} 1&0\\0&0\end{pmatrix} \begin{pmatrix} 1 & r\\0&1\end{pmatrix}) &= \begin{pmatrix} 1 & r^*\\0&1\end{pmatrix}\begin{pmatrix}
      1&r\\0&0
    \end{pmatrix}\\
    &= \begin{pmatrix}
      1&r\\0&0
    \end{pmatrix},
  \end{align*}
  which shows that $X\ot X(R_0)=R$.  A similar calculation shows that $Y\ot Y(R_0)=R'$.

  \Cref{prop:aut-qts} then implies that $X$ is an isomorphism of ribbon Hopf algebras
  \[(\D(G),R_0,u_{R_0})\to (\D(G),R,u_R),\]
  and that $Y$ is an isomorphism of ribbon Hopf algebras
  \[(\D(G),R_0,u_{R_0})\to (\D(G),R',u_{R'}).\]
\end{example}

More generally, we have the following description of the orbit of $R_0$.
\begin{prop}\label{prop:R0-orbit}
  Let $G$ be a purely non-abelian finite group and let $R=\cmorph$ be a quasitriangular structure of $\D(G)$.  Then $R\sim R_0$ if and only if $u^*\in\Aut_c(G)$ and $v=pu\inv r$.
\end{prop}
\begin{proof}
  By \cref{prop:aut-qts}, $R\sim R'$ if and only if
  \begin{align*}
    u &= \alpha\delta^*,\\
    r &= \alpha\beta^*,\\
    p &= \gamma\delta^*,\\
    v &= \gamma\beta^*,
  \end{align*}
  for some \[X=\begin{pmatrix}
    \alpha&\beta\\\gamma&\delta
  \end{pmatrix}\in\Aut(\D(G)).\]
  By \cref{thm:aut-dg} we have $\alpha^*,\delta\in\Aut(G)$ and $u^*=\delta\alpha^*\in\Aut_c(G)$.  Then given $\alpha,\delta$ we can always solve $r=\alpha\beta^*$ and $p=\gamma\delta^*$ for $r,p$ or $\beta,\gamma$ when given the other two.  Indeed, we have
  \begin{align*}
    v &= \gamma\beta^*\\
    &= p\delta^{*-1}\alpha\inv r\\
    &= p(\alpha\delta^*)\inv r\\
    &= pu\inv r.
  \end{align*}
  This completes the proof.
\end{proof}
In the case when $G$ has an abelian direct factor, the main obstruction is that it need no longer be the case that $\alpha^*,\delta\in\Aut(G)$ or that $\delta\alpha^*\in\Aut_c(G)$.  Indeed, when $G$ is abelian then $\alpha,\delta$ can both be trivial, since $\gamma,\beta$ can then be isomorphisms. Nevertheless, since $\Aut(\D(G))$ is known even in the purely non-abelian case \citep{KS14,LenPri:Brauer}, one can in principle always determine the equivalence class of $R_0$.

\begin{cor}
  Let $G$ be a purely non-abelian group. Then there are precisely $|\Aut_c(G)|\cdot |\BCh{G}|\cdot |\End(Z(G))|$ quasitriangular structures of $\D(G)$ which are equivalent to $R_0$.
\end{cor}
\begin{proof}
  By \citep[Theorem 6.7]{K14} we have \[|\Aut(\D(G))|=|\Aut(G)|\cdot |\Aut_c(G)|\cdot |\BCh{G}|\cdot|\End(Z(G))|\]
  and also by \citep[Example 9.6]{K14} that the stabilizer of $R_0$ is (isomorphic to) $\Aut(G)$.  Thus the size of the orbit follows.  Alternatively, the order also follows directly from \cref{prop:R0-orbit}.
\end{proof}
Indeed, for such a $G$ it follows from \cref{thm:qts} that there are precisely
\[ |\Aut_c(G)|\cdot |\BCh{G}| \cdot |\End(Z(G))| \cdot |\Hom(G,Z(G))|\]
quasitriangular structures $\cmorph$ of $\D(G)$ with $u^*\in\Aut_c(G)$. Thus there are quasitriangular structures with $u^*\in\Aut_c(G)$ which are inequivalent to $R_0$ if and only if $\Hom(G,Z(G))$ is non-trivial, or equivalently that $\gcd([G:G'],|Z(G)|)\neq 1$.

In general, it need not be the case that every $v\in\Hom(G,Z(G))$ appears in a quasitriangular structure $\cmorph$ that is equivalent to $R_0$.  At the extreme end, sometimes $v$ must in fact always be trivial.
\begin{cor}\label{cor:z-not-dupe}
  Let $G$ be a finite group.  Then the following are equivalent.
  \begin{enumerate}
    \item $G$ is a stem group.
    \item Every quasitriangular structure $\cmorph$ of $\D(G)$ with \[\cmorph\sim R_0\] has $v$ trivial.
  \end{enumerate}
\end{cor}
\begin{proof}
  From \cref{eq:aut-qts-v} if $\cmorph\sim R_0$ then $v=\gamma\beta^*$.  That this composition is trivial for all choices of $\gamma,\beta$ is equivalent to $Z(G)\subseteq G'$, which is the definition of a stem group.
\end{proof}
\begin{example}\label{ex:dihedral-z}
    $G=D_8$, the dihedral group with 8 elements, is a stem group with $\Hom(G,Z(G))$ non-trivial.  Thus for $R=\cmorph[1][0][0][z]$ with $z\in\Hom(G,Z(G))$ non-trivial the ribbon Hopf algebras $(\D(G),R_0,u_{R_0})$ and $(\D(G),R,u_R)$ are non-isomorphic.
\end{example}
\begin{example}
  On the other hand, \cref{prop:R0-orbit} also says that any quasitriangular structure of the form $\cmorph[1][r][p][0]$ with $pr\neq 0$ is necessarily not equivalent to $R_0$.  The preceding corollary also guarantees that such choices for $p,r$ always exist when $G$ is not a stem group.
\end{example}

\section{Factorizability}\label{sec:factor}
A quasitriangular Hopf algebra $(H,R)$ is said to be factorizable if $\tau(R)R$ is a non-degenerate tensor on $H\ot H$.  Equivalently, if the linear map $H^*\to H$ given by $f\mapsto (m\circ(f\ot\id))(\tau(R)R)$ is bijective.  We conclude the paper by considering this property for quasitriangular structures of $\D(G)$.  We need a number of basic definitions and lemmas to proceed.

As usual, we fix a finite group $G$ and a decomposition \[G=G_0\times G_1\times\cdots G_n,\] where $G_0$ is abelian and $G_1,...,G_n$ are indecomposable non-abelian groups.
\begin{df}\label{def:factor-defs}
  Let $E\subseteq\{1,...,n\}$.  We make the following definitions.
  \begin{enumerate}
    \item $\pi_E$ is the canonical retraction to the subgroup $\prod_{i\in E} G_i$.  This is given by the canonical surjection $G\to\prod_{i\in E} G_i$ followed by the canonical injection $\prod_{i\in E} G_i\to G$.
    \item $\pi_{E^c}$ is the canonical retraction to the subgroup $\prod_{0<i\not\in E} G_i$.
    \item Given $x\in G$ we may uniquely write $x= x_0 x_E x_{E^c}$ where $x_0\in G_0$, $x_E=\pi_E(x)$, and $x_{E^c}= \pi_{E^c}(x)$.
    \item We define a group
    \[ G_E = G_0 \times \pi_E(G) \times (\pi_{E^c}(G))^{\op}.\]
    This group is canonically isomorphic to $G$.
  \end{enumerate}
\end{df}

\begin{lem}\label{lem:twist-basis}
  For any $E\subseteq\{1,...,n\}$ the following is a basis of $\DDG$:
  \[ \{ (\pi_E(x)\rhup e_y)\# (\pi_{E^c}(y)\rhup x) \ | \ x,y\in G\}.\]
\end{lem}
\begin{proof}
  Indeed, we can show that this is precisely the standard basis of $\DDG$.  So for $g,h,x,y\in G$ we have
  $(\pi_E(x)\rhup e_y)\# (\pi_{E^c}(y)\rhup x) = e_g \# h$ if and only if, in the notation of \cref{def:factor-defs}, the following all hold
  \begin{align*}
    y_0 &= g_0;& y_{E^c}&=g_{E^c};& y_E&=g_E^{x_E}=g_E^{h_E};\\
    x_0 &= h_0;& x_E &= h_E;& x_{E^c}&=h_{E^c}^{y_{E^c}}=h_{E^c}^{g_{E^c}}.
  \end{align*}
  Thus $g,h$ uniquely determine $x,y$ and conversely, and the desired claim follows.
\end{proof}

\begin{prop}\label{prop:fact-morph}
  Let $R=\cmorph$ be a quasitriangular structure of $\D(G)$ associated to the set $E\subseteq \{1,...,n\}$ as in the proof of \cref{thm:qts}.  Then we have a well-defined morphism
  \begin{align}\label{eq:fact-morph}
  \Psi_R=\begin{pmatrix}
    v^*+u&r^*+r\\
    p^*+p&u^*+v
  \end{pmatrix}\in\End(\D(G_E)).
  \end{align}
\end{prop}
\begin{proof}
  The definition of $G_E$ in \cref{def:factor-defs} and the properties of the components of $R$ in \cref{thm:qts} combined with the results of \citep{ABM,K14} show that $\Psi_R$ is a well-defined endomorphism of $\D(G_E)$.
\end{proof}
Since $G\cong G_E$ canonically, we can always identify the underlying vector spaces of $\D(G)$ and $\D(G_E)$.  We do so whenever convenient without further mention.

\begin{example}
  For $R_0$ we have that $\Psi_{R_0}$ is the identity of $\D(G)$.  In particular, $\Psi_{R_0}$ is an automorphism, and $R_0$ is well-known (and easily shown) to be factorizable.

  Similarly, for $R_1$ we have that $\Psi_{R_1}$ is the identity of $\D(G^{\op})$. Thus $\Psi_{R_1}$ is an automorphism, and $R_1$ is well-known to be factorizable.
\end{example}

Indeed, we can now complete our goal by showing that this relation between $\Psi_R$ and the factorizability of $R$ holds for arbitrary $R$.

\begin{thm}
  Let $R=\cmorph$ be a quasitriangular structure of $\D(G)$ associated to the set $E\subseteq \{1,...,n\}$ as in the proof of \cref{thm:qts}. Define $\Psi_R$ as in \cref{prop:fact-morph}.

  Then $F(\tau(R)R)$ and $\Psi_R$ have the same image. Therefore $R$ is factorizable if and only if $\Psi_R\in\Aut(\D(G_E))$.
\end{thm}
\begin{proof}
  By \cref{thm:qts} we have that $u^*,v$ are normal group endomorphisms such that $u^*\comm v$. With this we can easily compute that
  \begin{align*}
    \tau(R)R &= \sum_{s,t,a,b\in G} (v(t)\rhup e_s)\bowtie ab v(t)u^*(t)\\
    &\qquad\qquad{}\ot r^*(s)r(s)(u^*(s)\rhup e_t)\bowtie p^*(e_a)p(e_b)u^*(s)v(s)\\
    &= \sum_{s,t,a,b\in G} e_s\bowtie ab u^*(t^s)v(t)\ot r^*(s)r(s) e_t \bowtie p^*(e_a)p(e_b)u^*(s)v(s^t).
  \end{align*}
From this we then have
  \begin{align}\label{eq:fact-1}
  \begin{split}
    F(\tau(R)R)(f\# x) &= \sum_{t\in G} r^*(x)r(x)f\com3(u^*(t^x)v(t))e_t\\
    &\qquad\qquad{}\bowtie p(f\com1)p^*(f\com2)u^*(x)v(x^t).
  \end{split}
  \end{align}
Note that $F(\tau(R)R)$ is in general not a morphism of Hopf algebras.

Let $E\subseteq\{1,...,n\}$ be the set associated to $R$ for this decomposition as in the proof of \cref{thm:qts}. Now \cref{def:factor-defs,eq:fact-1,thm:qts} show that
  \begin{align}\label{eq:fact-2}
  \begin{split}
    F&(\tau(R)R)((\pi_E(x)\rhup f)\#x)\\
     &= \sum_{t\in G} r^*(x)r(x)(u+v^*)(f\com 1)e_t\bowtie (p+p^*)(f\com2)u^*(x)v(x^t).
  \end{split}
  \end{align}
Next we consider the special case $f=e_y$ for some $y\in G$.  Then \cref{eq:fact-2} becomes
  \begin{align}\label{eq:fact-3}
  \begin{split}
    F(\tau(R)R)&((\pi_E(x)\rhup e_y)\#x)\\
    &= \sum_{\substack{c,t\in G\\u^*(t)v(t)c=y}} r^*(x)r(x)e_t\bowtie (p+p^*)(e_c)u^*(x)v(x^t).
  \end{split}
  \end{align}
Let $C,D$ be the subgroups of $Z(G)$ determined by $p+p^*$; in particular, $p+p^*$ gives an isomorphism $\widehat{C}\to D$, and in the above summation we must have $c\in C$ for the term to be non-zero.

As we have noted, $u^*\comm v$ and $v$ is normal.  So if $u^*(t)v(t)c=y$ for some $c\in C$ and $t,y\in G$ then
  \begin{align}\label{eq:fact-4}
  \begin{split}
    v(x^t) &= v(x^{v(t)})\\
    &= v(x^{u^*(t)v(t)c})\\
    &= v(x^y).
  \end{split}
  \end{align}
So \cref{eq:fact-4,eq:fact-3,def:factor-defs,prop:fact-morph,thm:qts} combine to give
  \begin{align}\label{eq:fact-5}
  \begin{split}
    F(\tau(R)R)(&(\pi_E(x)\rhup e_y)\#((\pi_{E^c}(y)\rhup x))\\
    &= \sum_{\substack{c,t\in G\\u^*(t)v(t)c=y}} r^*(x)r(x)e_t\bowtie (p+p^*)(e_c)u^*(x)v(x)\\
    &= \Psi_R(e_y\bowtie x).
  \end{split}
  \end{align}
By \cref{lem:twist-basis} we therefore conclude that $F(\tau(R)R)$ and $\Psi_R$ have the same image.  Thus $F(\tau(R)R)$ is bijective if and only if $\Psi_R$ is bijective.

This completes the proof.
\end{proof}

\begin{cor}
  Suppose $G$ is a purely non-abelian group.  Then every quasitriangular structure of $\D(G)$ is factorizable.
\end{cor}
\begin{proof}
  Let $R=\cmorph$ be a quasitriangular structure of $\D(G)$ associated to the set $E$.  By \cref{thm:qts} $u^*+v\in\Aut_c(G_E)$, so by \cref{thm:aut-dg} $\Psi_R\in\Aut(\D(G_E))$ and the result follows from the preceding theorem.
\end{proof}


\bibliographystyle{plainnat}
\bibliography{../references,../refs-quasitriangular}
\end{document}